\documentclass[12pt]{amsart}
\usepackage{bookman}
\usepackage[pagebackref]{hyperref}
\usepackage[alphabetic,backrefs,lite,nobysame]{amsrefs}
\usepackage{amsmath}
\usepackage{url,amssymb,enumerate,colonequals,bm}
\usepackage[headings]{fullpage}
\usepackage[all]{xy}
\usepackage{comment}

\numberwithin{equation}{section}
\numberwithin{figure}{section}

\usepackage{mathrsfs}% for \mathscr (script letters)
\usepackage{tikz-cd}

\usepackage{calrsfs}

\newcommand{\ab}{{\operatorname{ab}}}

\newcommand{\calL}{{\mathcal L}}
\newcommand{\calA}{{\mathcal A}}

\newcommand{\Z}{{\mathbb Z}}

\newcommand{\F}{{\mathbb F}}

\newcommand{\ord}{\text{ord}}

\newcommand{\calS}{{\mathcal S}}

\newcommand{\pp}{{\mathfrak p}}

\newtheorem*{theorem*}{Theorem}
\newtheorem*{conjecture*}{Conjecture}
\newtheorem{conjecture}[equation]{Conjecture}
\newtheorem*{question*}{Question}

\newtheorem{theorem}[equation]{Theorem}
\newtheorem{lemma}[equation]{Lemma}
\newtheorem{corollary}[equation]{Corollary}

\newtheorem{proposition}[equation]{Proposition}

\newtheorem*{mainthm}{Theorem \ref{thm:main}}

\theoremstyle{definition}
\newtheorem{definition}[equation]{Definition}
\newtheorem{question}[equation]{Question}

\newtheorem{remark}[equation]{Remark}

\theoremstyle{remark}

\def\a{\mathfrak{a}}
\def\d{\mathfrak{d}}
\def\e{\varepsilon}

\newcommand{\oo}{\mathcal{O}}
\newcommand{\Q}{\mathbb{Q}}

\newcommand{\Qbar}{\overline{\mathbb{Q}}}

\newcommand{\rank}{\mathrm{rank}\;}

\def\image{\mathrm{image}}

\makeatother

\title{Defining $\Z$ using unit groups}
\author{Barry Mazur}
\address{Department of Mathematics\\ Harvard University \\ Cambridge, MA 02138-2901}
\email{mazur@g.harvard.edu}
\urladdr{http://www.math.harvard.edu/~mazur}
\author{Karl Rubin}
\address{Department of Mathematics\\
UC Irvine\\
Irvine, CA 92697, 
USA}
\email{krubin@uci.edu}
\urladdr{https://math.uci.edu/~krubin}
\author{Alexandra Shlapentokh}
\address{Department of Mathematics \\ East Carolina University \\ Greenville, NC 27858}
\email{shlapentokha@ecu.edu }
\urladdr{myweb.ecu.edu/shlapentokha}


\dedicatory{To Henryk Iwaniec on his 75th birthday, with admiration}

\begin{document}

\begin{abstract}
We consider first-order definability and decidability questions over 
rings of integers of algebraic extensions of $\Q$, paying attention 
to  the uniformity of definitions.   
The uniformity follows from the simplicity of our first-order definition of $\Z$.
Namely, we prove that for a large collection of algebraic extensions $K/\Q$,
$$
\{x \in \oo_K : \text{$\forall \e \in \oo_K^\times \;\exists \delta \in \oo_K^\times$ 
such that $\delta-1 \equiv (\e-1)x \pmod{(\e-1)^2}$}\} = \Z
$$
where $\oo_K$ denotes the ring of integers of $K$.  One of the corollaries of our results is undecidability of the field of constructible numbers, a question posed by Tarski in 1948.
\end{abstract}

\thanks{The authors would like to thank Caleb Springer for his careful reading of their paper and the anonymous Referee for very helpful comments and for pointing out the implication of our results for a question of Tarski.\\
The research for this paper was partially supported by DMS grants 2152098 (AS), 2152149 (BM), 2152262 (KR).
AS was also partially supported by an ECU Creative Activity and Research Grant during the summer of 2022.}

\maketitle
\setcounter{tocdepth}{1}

\tableofcontents

\section{Introduction}
We consider first-order definability and decidability questions over rings and fields of integers of algebraic  extensions of $\Q$, paying attention to  the uniformity of the manner in which definability is established.  

Let $K$ be an algebraic, possibly infinite, extension of $\Q$ and let $\oo_{K}$ be the ring of integers of $K$. 

\begin{question}
\label{q1}
Is there a first-order definition of $\Z$ over $\oo_{K}$ in the language of rings?
\end{question}

\begin{question}
\label{q2}
Is the first-order theory of $\oo_{K}$ in the language of rings undecidable?
\end{question}

\begin{question}
\label{q3}
Is the first-order theory of $K$ in the language of rings undecidable?
\end{question}

It is well-known that if $\oo_K$ is definable over $K$ and the first-order theory of $\oo_K$ is undecidable, then the first-order theory of $K$ is undecidable.  In this paper we concentrate on the first-order theory of rings of integers of infinite algebraic extensions of $\Q$ and then use some existing results on definability of the rings of integers of these fields over the fields to obtain undecidability results for fields.  

In particular, we answer a question of A. Tarski concerning (un)decidability of the first-order theory of the field of constructible numbers (see \cite{Tar48}, Notes \#10 and see Corollary \ref{cor:Tarski} below) and a question posed by C. Martinez-Ranero, J. Utreras and C. Videla concerning undecidability of $K^{(d)}$, the compositum of all extensions of degree less or equal to a fixed integer $d$ of a number field $K$ (see \cite{MarUtrVid}).

\subsection{The language of rings}
The first-order language $(0,1,+, \times)$ of rings is essentially the language of polynomial equations. For a fixed ring $R$, a sentence in this language can be shown to be equivalent to one that has the form:
\begin{equation}
\label{eq:sent}
E_1x_1\ldots E_rx_r \bigwedge_{i=1}^{\ell}P_i(x_1,\ldots,x_r)=0, 
\end{equation}
where each $E_i$ is either the universal or the existential quantifier ranging over $R$ and each $P_i$ is a polynomial with coefficients in $\Z$.

The first-order theory of a countable ring $R$ is the set of all sentences 
of the form \eqref{eq:sent} that are true over $R$.

   Using this language we can define subsets of our ring:
\begin{equation}
\label{eq:def}
A=\{a \in R|E_1x_1\ldots E_rx_r \bigwedge_{i=1}^{\ell} P_i(a,x_1,\ldots,x_r)=0\}, 
\end{equation}
where as above each $E_i$ is either the universal or the existential quantifier 
ranging over $R$ and $P_i$ are polynomials over $\Z$.  That is,
an element $a \in R$ is in $A$ if and only if the sentence 
$E_1x_1\ldots E_rx_r \bigwedge_{i=1}^{\ell}P_i(a,x_1,\ldots,x_r)=0$ is true over $R$. 
If \eqref{eq:def} holds for a set $A$ we say that the set $A$ is {\em (first-order)  definable over $R$}.

 If the sentence \ref{eq:sent}  has existential quantifiers only, then it is in what is called {\em the existential language of rings} and  we say that the set $A$ satisfying \ref{eq:def}  is {\em existentially  definable over $R$}.   The problem of deciding whether a sentence is true in this language over a ring $R$ is essentially Hilbert's Tenth Problem for the ring.

Sometimes we expand the language of rings by adding countably many constant symbols to allow the coefficients of the polynomials to be in $R$. 

Given two rings $R_1 \subset R_2$ such that $R_1$ has an undecidable first-order theory and has a first-order definition over $R_2$ as in \eqref{eq:def}, we can conclude that the first-order theory of $R_2$ is undecidable.   For, if $R_2$ were decidable, combining the algorithm giving first-order decidability in $R_2$ with the first-order definition of $R_1$ over $R_2$ would then give us an algorithm for decidability of $R_1$,   a contradiction.  

Applying this reasoning wiith $R_1 = \Z$ shows that if the answer to Question \ref{q1} above is ``yes'', then the answer to Question \ref{q2} is ``yes'' as well. 

\subsection{Uniform definitions}

\begin{definition}[Uniform definitions over rings of algebraic integers]
\label{def:uni}
Let $K$ be a field of algebraic numbers, $\calL$  a collection of extensions 
of $K$ in the algebraic closure ${\bar K}$, and $A$ a subset of $\oo_{\bar K}$. 
We say that {\em $A$ has a uniform first-order definition over rings of integers of 
all fields in $\calL$}  if there is a single polynomial 
$P(t,\bar x) \in \oo_{K}[t,\bar x]$  in $1+r$ variables 
($t$ and ${\bar x} = (x_1,x_2,\dots x_r)$, for some $r \in \Z_{\ge 1}$) 
and  a first-order formula 
\[
Q(t):=E_1x_1\ldots E_rx_rP(t,\bar x), 
\] 
where $E_i$ is either a universal or an existential quantifier and $t$ is the only free variable, 
such that $Q(t)$ is a first-order definition of $A \cap \oo_L$  over $\oo_L$ for every $L\in \calL$. 
\end{definition}

\begin{remark} 
Since $\Z$ is existentially undecidable, if $\Z$ has a uniform first-order 
definition over rings of integers of  all fields in $\calL$, then we have 
a single formula giving us  a ``uniform'' way of showing that these rings 
of integers are first-order undecidable.  
\end{remark}

\subsection{Brief history}
The first result showing that a first-order theory of a ring is undecidable 
was due to J.B. Rosser \cite{Rosser} building on results of A. Church \cite{Church} and K. G\"{o}del \cite{G}.  
The ring in question was $\Z$.  J. Robinson \cite{Rob1} produced a definition 
of $\Z$ for the rings of integers of every number field, thus showing that 
the first-order theory of these rings is undecidable.  Her definition was of 
a simple form $\forall\exists\ldots \exists$ but depended explicitly on the 
degree of the extension.  Later in \cite{Rob3} she produced a uniform 
definition of $\Z$ across all rings of integers  of number fields, but the 
uniform formula was much more complicated. (See Definition \ref{def:uni} 
for our interpretation of the notion of uniformity.)  
In this paper we construct a definition of $\Z$, without parameters, over the 
rings of integers of all number fields and some classes of infinite algebraic extensions of $\Q$. This definition has the form $\forall\forall\exists\ldots \exists$.

There are of course existential definitions of $\Z$ 
over rings of integers of some number fields and some infinite algebraic 
extensions of $\Q$ (see, for example, \cite{Den1}, \cite{Den2}, \cite{Den3}, 
\cite{Ph1}, \cite{Sha-Sh},  \cite{Sh2}, \cite{Vid89}, \cite{CPZ}, 
\cite{Sh36}, \cite{Sh33}, \cite{Sh37}, \cite{GFPast}, \cite{MRSh22}) and 
some well-known conjectures imply that $\Z$ has an existential definition 
over rings of integers of every number field (see \cite{MR}, \cite{MP18}, \cite{Pas22}).  
However, these definitions are not uniform across all number fields, and by 
Corollary \ref{thm:nonu} they cannot be.

While we now have a pretty good understanding of the first-order theory of rings of integers of number fields, the situation is quite different when it comes to infinite algebraic extensions of $\Q$.  For example there are infinite extensions whose ring of integers has decidable first-order theory (\cite{Rum1} and \cite{Dries4}), and other infinite extensions whose ring of integers has undecidable first-order theory. 

J. Robinson was the first person to produce examples of infinite algebraic extensions of $\Q$ where the first-order theory of the ring of integers is undecidable. She developed a method for constructing a first-order definition of $\Z$ for rings of totally real integers and a method for constructing  a first-order model of $\Z$ and thus proving undecidability for some rings of integers that are not totally real.  See for example \cite{V3}, \cite{VidVid1}, \cite{VidVid2}, \cite{GilRan}, \cite{CasVidVid}, \cite{MarUtrVid}, \cite{Sp}, \cite{ChecFehm},  \cite{Sp23} for results using J. Robinson's method applied to rings of totally real integers and rings of integers of extensions of degree 2 of totally real fields.  

The first important results showing that rings of integers of some infinite algebraic extensions of $\Q$ are first-order definable over their fields of fractions are due to C. Videla (see \cite{V2}, \cite{V3}).  These results were generalized by the third author in \cite{Sh40}.
\subsection{Totally real algebraic numbers versus totally real algebraic integers}
 The rings of integers and their fields of fractions do not always behave in the same fashion with respect to decidability/undecidability of their first-order theories.  A remarkable result  due to M. Fried, D. Haran and H. V{\"o}lklein in \cite{Fried3} shows that the first-order theory of the field of all totally real algebraic numbers is decidable. This field constitutes a boundary of sorts between the ``decidable'' and ``undecidable'', since J. Robinson showed in \cite{Rob3} that $\Z$ is first-order definable over the ring of all totally real integers and therefore the first-order theory of the ring of all totally real integers is undecidable. 

\subsection{ Non-big fields of algebraic numbers}
Our formula for defining $\Z$ is not dependent on whether the ring in question is real or not.  Rather it depends on whether the field in question is ``big'' or not (see \cite{MR20}).

\begin{definition}
\label{big} 
Let $K$ be a field of algebraic numbers.  
We say that $K$ is {\em  big } if 
$$
[K :{\mathbb  Q}] = \prod_{p\ {\rm prime}} p^\infty.
$$  
Equivalently,
$K$ is big if for every positive integer $n$, $K$ contains a number field $F$ 
with $[F : {\Q}]$ divisible by $n$.
\end{definition}

Our main theorem, which will be proved in \S\ref{mainpf}, is the following.

\begin{mainthm}
Let $\calA$ be the collection of all non-big fields of algebraic numbers.
There exists a first-order formula of the form 
``$\forall \forall\exists \ldots \exists$'' uniformly defining $\Z$ over the 
rings of integers of all fields in $\calA$.  
In particular the first-order theory of these rings can be uniformly shown to be first-order undecidable.
\end{mainthm}
\begin{definition}
Let $\{F_i\}, i \in \Z_{>0}$ be a sequence of fields such that $F_i \subset F_{i+1}$ and $F_{i+1}/F_i$ is Galois.  Then we will call the sequence $\{F_i\}$ a {\it  tower of Galois extensions}.

\end{definition}
\begin{corollary}
The following fields have an undecidable first-order theory.
\begin{enumerate}
\item A non-big union of a tower of Galois extensions of a number field. 
\item A non-big Galois extension of a number field. 
\item A union of a tower of extensions of degree less or equal to $d \in \Z_{>1}$.
\item A compositum of all extensions of degree less or equal to $d \in \Z_{>1}$ of a number field. 
\item The closure of a number field under taking of $n$-th roots or under taking $2n$-th roots of positive elements.
\end{enumerate}
\end{corollary}

\begin{proof}
If $M$ is a tower of Galois extensions  (or a Galois extension) of a number field $K$ and is not big, then for some prime $p$ the field $M$ is uniformly $p$-bounded 
in the sense of \cite{Sh40}, so $\oo_M$ has a first-order definition 
over $M$. (See \cite{Sh40} for the proof and note that in that paper the 
term ``$q$-bounded'' is used in place of ``$p$-bounded''.)
By Theorem \ref{thm:main}, we have that $\Z$ is first-order definable over $\oo_M$. 
So $\Z$ is first-order definable over $M$ and 
therefore the first-order theory of $M$ is undecidable. 

In the case of the last three assertions of the corollary, the field in question is clearly non-big and uniformly $p$-bounded for any $p>d$ in the case of the compositum or a tower of extensions of degree less or equal to $d$ and for any $p>2n$ in the case of the closure under taking $n$-th roots or $2n$-th roots.  The rest of the argument is the same as for the first two assertions of the corollary.

\end{proof}

The last item in the corollary implies the following corollary answering a question of Tarski.
\begin{corollary}
\label{cor:Tarski}
The first-order theory of the field of constructible numbers is undecidable.
\end{corollary}
\begin{proof}
The field of constructible numbers is smallest subfield of the field of algebraic numbers that is closed under the operation of taking a square root of  positive elements of the field.

\end{proof}

Our methods also provide a potential approach to proving undecidability 
of the ring of integers of the maximal abelian extension of $\Q$ 
(see Corollary \ref{c1}).
\subsection{Using units}
One of the main ideas behind our use of units is a simple algebraic identity $\frac{x^n-1}{x-1} \equiv n \bmod (x-1)$ in $\Z[x]$ for $x \ne 1$.  If $x$ is  a unit of a ring of integers, then so is $x^n$ leading to our Definition \ref{def:1}.

The first use of units exploiting this algebraic identity for the purposes of (existential) definability in the form of Pell equation belongs to J. Robinson and M. Davis (see \cite{Dav73}). The Pell equation over $\Z$ is after all just a norm equation of the units of a quadratic extension.  Later, J. Denef and L. Lipshitz  adapted this idea for algebraic extensions (see \cite{Den1}, \cite{Den3},\cite{Den2}).  Following their lead T. Pheidas (\cite{Ph1}), C. Videla (\cite{Vid89}) and the third author (\cite{Sh2}) used this particular unit method for another collection of number fields.  The third author used units, $S$-units and more generic norm equations of units to define $\Z$ over big subrings of number fields and infinite algebraic extensions of $\Q$ (see for example \cite{Sh94}, \cite{Sh02}, \cite{Sh04}).

C. Martinez-Ranero, J. Utreras and C. Videla in \cite{MarUtrVid} and C. Springer in \cite{Sp} and \cite{Sp23} also make use of units of the field under consideration but their methods are entirely different from ours.

\section{First-order passage from submonoids to subrings}

For any algebraic extension $K$ of $\Q$ let $\oo_K$ denote the ring of integers of $K$ 
and $U_K := \oo_K^\times$ the group of units of $\oo_K$.

\begin{definition}
\label{def:1}
Let $K$ be a field of algebraic numbers, and $M\subset \oo_{K}$ a multiplicative monoid.   
Define
\begin{align*}
R_{K, M}:=&\{x \in \oo_{K} : \text{$\forall \e \in M \; \exists \delta \in M \; 
\exists w \in \oo_{K}$ such that $\delta-1 = (\e-1)x + (\e-1)^2w$}\} \\
=&\{x \in \oo_{K} : \text{$\forall \e \in M \; \exists \delta \in M$ 
such that $\delta-1 \equiv (\e-1)x \pmod{(\e-1)^2}$}\}.
\end{align*}
When $M = U_K$ we will denote $R_{K,U_K}$ simply by $R_K$.
\end{definition} 

\begin{lemma}
\label{sumform}
Suppose $\e, x_1, x_2,\ldots,x_n, \delta_1,\ldots,\delta_n \in \oo_K$ 
for some $n \ge 1$, and for every $i$ 
$$
\delta_i-1 \equiv (\e-1)x_i \pmod{(\e-1)^2}.
$$
Then 
$$
\prod_i\delta_i - 1 \equiv (\e-1)\sum_i x_i \pmod{(\e-1)^2}.
$$
\end{lemma}

\begin{proof}
By induction it is enough to prove the lemma for $n = 2$.  
Since $\delta_i-1 \equiv 0 \pmod{(\e-1)}$, when $n = 2$ the lemma follows from 
the congruence
\begin{multline*}
\delta_1\delta_2 - 1 = (\delta_1 - 1) + (\delta_2 - 1) + (\delta_1 - 1)(\delta_2 - 1) \\
   \equiv (\delta_1 - 1) + (\delta_2 - 1) \equiv (\e-1)x_1 +(\e-1)x_2 \pmod{(\e-1)^2}.
\end{multline*}
\end{proof}

\begin{proposition}
\label{prop:isaring} 
Let $M$ be a monoid in $\oo_{K}$, and let $\Q(M)$ be the field generated over $\Q$ by $M$.  Then:
\begin{enumerate} 
\item 
$R_{K, M}$ is closed under addition and multiplication  (i.e., is a `semi-ring').
\item   
If  $M$ is a subgroup of $U_{K}$ then  $R_{K, M}$ is a subring of $\oo_{K}$.
\item 
 If $M$ contains a unit of infinite order, then $R_{K, M} \subset \oo_{\Q(M)}$.
\end{enumerate}
\end{proposition}

\begin{proof}
That $R_{K, M}$ is closed under addition follows directly from Lemma \ref{sumform}.

Suppose $x, y \in R_{K,M}$ and $\e \in M$.
Since $x \in R_{K, M}$, there is a $\delta \in M$ such that 
$$
\delta - 1 \equiv (\e-1)x \pmod{(\e-1)^2}.
$$
Since $y \in R_{K, M}$, there is a $\delta' \in M$ such that
$$
\delta' - 1 \equiv (\delta-1)y \equiv ((\e-1)x)y \pmod{(\e-1)^2}.
$$
It follows that $xy \in R_{K, M}$, which completes the proof of (1).

Now suppose $M$ is a subgroup of $U_K$.  
That $R_{K,M}$ is a ring will follow from (1) once we show that $-1 \in R_{K,M}$.  
Suppose $\e \in M \subset U_{K}$.  Since $M$ is a group, $\e^{-1} \in M$, and 
$$
\e^{-1}-1 = -(\e-1) + \e^{-1}(\e-1)^2 \equiv -(\e-1) \pmod{(\e-1)^2}.
$$
It follows that $-1 \in R_{K, M}$, which completes the proof of (2).

Finally, suppose that $M$ contains a unit $\e$ of infinite order.  If $\a$ is a nonzero ideal 
of $\oo_{\Q(\e)}$, then the reduction of $\e$ has finite order in $(\oo_{\Q(\e)}/\a)^\times$,
i.e., $\e^n-1 \in \a$ for some $n \ge 1$.  Therefore
\begin{equation}
\label{Y}
\bigcap_{n \ge 1} (\e^n-1)\oo_{K} \subset 
   \bigcap_{\text{nonzero ideals $\a$ of $\oo_{\Q(\e)}$}} \a\oo_{K} = 0.
\end{equation}

Suppose $x \in R_{K, M}$, and $n \in \Z_{>0}$.  Since $\e^n \in M$, 
there is a $\delta \in M$ such that
$$
\delta-1 \equiv x(\e^n-1) \pmod{(\e^n-1)^2}.
$$ 
If $y$ is any conjugate of $x$ over $\Q(M)$, then the same congruence holds with $x$
replaced by $y$, and since $\e^n \ne 1$ we get
$$
x-y \in (\e^n-1)\oo_K.
$$
Thus $x-y \in \cap_{n>0} (\e^n-1)\oo_K$, 
which is zero by \eqref{Y}.  Hence $y = x$, and it follows that $x \in \oo_{\Q(M)}$.
\end{proof}

\begin{remark}
Suppose $M$ is a subgroup of $U_K$.
A priori, the ring $R_{K, M}$ is not necessarily the full ring of 
integers  ${\tilde{R}_{K, M}}$ in its field of fractions.  
But ${\tilde{R}_{K, M}}$ admits an existential definition in terms of $R_{K, M}$:
\end{remark}

\begin{proposition}
If $M$ is a subgroup of $U_K$, then
$${\tilde{R}_{K, M}}=\{x \in \oo_{K} \,|\, \exists y, z \in R_{K, M}, y \ne 0, xy=z\}.$$
\end{proposition}

\begin{proof}  
Note that the set $\{y \in \oo_K : y \ne 0\}$ has an existential definition by a result of 
Denef (\cite[Lemma 9]{Den3}; see also \cite[Lemma 4.6]{MRSh22}).

Let $R$ denote the right hand side of 
the equation in the statement of the proposition.  
Then $R \subset {\tilde R}_{K, M}$.

For every number field $F$ contained in the field of fractions of $R_{K, M}$, 
there is an element $\alpha \in R_{K, M}$ such that $\Q(\alpha)=F$. For every 
$\beta \in \oo_{F}$ there exists $m \in \Z$ such that 
$m\beta \in \Z[\alpha] \subset R_{K, M}$, and taking $x = \beta$, $y=m$, and $z = m\beta$ 
in the definition of $R$ we see that $\beta \in R$. Thus $R$ contains 
every algebraic integer of $F$. Since this is true for every number field $F$ in 
the field of fractions of  ${\tilde{R}_{K, M}}$,  it follows that 
$R$ is integrally closed in its field of fractions, i.e., 
$R={\tilde{R}_{K, M}}$.  
\end{proof} 

\begin{corollary} 
If $M$ is a subgroup of $U_K$, and $M$ admits a first-order definition in $\oo_{K}$, 
then the rings $R_{K, M}$ and ${\tilde{R}_{K, M}}$ admit first-order definitions 
in $\oo_{K}$ as well.
\end{corollary}

\begin{lemma}
\label{norm}
Let $L/K$ be a finite extension.  
Suppose $\e \in U_{K}$, $x \in \oo_{K}$, and $\delta \in U_{L}$ satisfy 
\[
(\e-1)x \equiv \delta-1 \pmod {(\e-1)^2\oo_{L}}.
\]
  Then 
 \[
 (\e-1)[L:K]x \equiv {N_{L/K}(\delta)-1} \pmod{(\e-1)^2\oo_{K}}.
 \]
 \end{lemma}
 
 \begin{proof} 
 Let $\Sigma$ be the set of embeddings of $L/K$ into the Galois closure $\tilde{L}$ of $L/K$. 
 Since $x$ and $\e$ lie in $K$, for every $\sigma \in \Sigma$ we have that
 \[
(\e-1)x \equiv \sigma(\delta)-1 \pmod{(\e-1)^2\oo_{\tilde{L}}}.
\]
We have $\prod_{\sigma\in\Sigma} \sigma(\delta) = N_{L/K}(\delta)$ and
$|\Sigma| = [L:K]$.  Therefore Lemma \ref{sumform} applied with $n = [L:K]$, the given $\e$, 
all $x_i$ equal to the given $x$, and the $\delta_i$ equal to  
$\sigma(\delta)$ for $\sigma \in \Sigma$ shows that 
$$
{N_{L/K}(\delta)-1} \equiv (\e-1)[L:K]x  \pmod{(\e-1)^2\oo_{\tilde{L}}}.
$$
Both sides of this congruence belong to $\oo_K$, so the congruence holds
 modulo $(\e-1)^2\oo_K$.
\end{proof}

Recall that we denote $R_{K,U_K}$ simply by $R_K$.

\begin{lemma}
\label{le:smaller}
Let $W$ be a subgroup of $U_K$, let $n \in \Z_{>0}$ and let $\a$ be an ideal of $\oo_K$.  Let
$$
M := \{u^n : u \in W, u \equiv 1 \pmod{\a}\}.
$$
Then $nR_{K,W} \subset R_{K,M}$.  
\end{lemma}

\begin{proof}
Let $x \in R_{K,W}$.  We will show $nx \in R_{K,M}$.  Let $\e \in W$.  
Since $x \in R_{K,W}$ there exists $\delta \in W$ such that 
\[
(\e-1)x \equiv \delta-1 \pmod{(\e-1)^2}
\]
 and it follows that $\delta \equiv 1 \pmod{\a}$ and $\delta^n \in M$.  
 Therefore by Lemma \ref{sumform} 
 applied with this $n$ and $\e$, all $x_i := x$, and all $\delta_i := \delta$,
 \[
 (\e-1)nx \equiv \delta^n-1 \pmod{(\e-1)^2}.
 \]
Thus $nx \in R_{K,M}$.
\end{proof}

\begin{definition}
We call a field $K$ of algebraic numbers a {\em CM-field} if $K$ is a 
totally imaginary quadratic extension of a totally real field.  We 
include the case $[K:\Q] = \infty$.
\end{definition}

\begin{proposition}
\label{cor:real1}
Let $K$ be a CM-field that is not an imaginary quadratic field, 
and let $F$ be the maximal real subfield of $K$.  
Then $R_{K} \subset \oo_F$.
\end{proposition}

\begin{proof}
Fix $d \in \Z_{>2}$ and let $M := \{u^2 : u \in U_K, u \equiv 1 \pmod{d\oo_K}\}$.  
We first claim that $M \subset U_{F}$.  Let $\alpha \mapsto \bar\alpha$ 
denote complex conjugation on $K$.
If $\e \in M$, we can write $\e = u^2$ with $u \in U_K$, 
$u \equiv 1 \pmod{d}$. Thus 
$$
\e = u^2 = (u\bar{u})(u/\bar{u}).
$$
Since $u/\bar{u} \in U_K$ and all of its archimedean absolute values are $1$, 
$u/\bar{u}$ is a root of unity.  
Since $u \equiv 1 \pmod{d}$ we must have $u/\bar{u} = 1$. 
Thus $\e = u\bar{u} \in U_F$.

Since $K$ is not an imaginary quadratic field, $M$ contains units of infinite order. 
Thus, by Lemma \ref{le:smaller} for the first containment and 
Proposition \ref{prop:isaring}(3) for the second, we have
$$
2 R_{K} \subset R_{K,M} \subset \oo_{\Q(M)} \subset \oo_F.
$$
It follows that $R_K \subset \oo_F$.
\end{proof}

\begin{remark}
We can apply Proposition \ref{cor:real1} to $\Q^{\ab},$ the maximal abelian 
extension of $\Q$, to conclude that there is a totally real ring 
(namely, $R_{\Q^\ab}$) that is first-order definable over $\oo_{\Q^{\ab}}$.  
Unfortunately, at the moment we cannot identify the ring $R_{\Q^\ab}$.
\end{remark}

\begin{remark}
Let $\Q{^{\text{ab},+}}$ be the the maximal totally real subfield of $\Q^{\ab}$. 
It is easy to see that the proof of J. Robinson from \cite{Rob3} showing undecidability of the ring of integers of the field of totally real numbers applies to the ring of integers of  $\Q{^{\text{ab},+}}$. So, the first-order theory of $\oo_{\Q{^{\text{ab},+}}}$ is 
undecidable.  Moreover, she  conjectured that the ring of integers of any 
totally real field is undecidable.  J. Koenigsmann \cite{Ko14} conjectured that 
$\oo_{\Q{^{\text{ab},+}}}$ has a first-order definition over $\oo_{\Q^{\text{ab}}}$ 
implying  via the above mentioned result  that $\oo_{\Q^{\ab}}$ also 
has an undecidable first-order theory. As a corollary of Proposition \ref{cor:real1} 
we can add the following conditions that would 
imply undecidability of the first-order theory of $\oo_{\Q^{\ab}}$.
\end{remark}

\begin{corollary}
\label{c1}
The ring $\oo_{\Q^{\text{ab}}}$ is first-order undecidable if at least one of the following
conditions is satisfied:
\begin{enumerate}
\item the ring $R_{\Q^{\text{ab}}}$ is undecidable, or
\item 
the ring of integers of every real abelian extension of $\Q$ is undecidable, or
\item 
there exists a non-big subfield of $\Q^{\ab}$ whose ring of integers is 
first-order definable over $\oo_{\Q{^{\ab}}}$.
\end{enumerate}
\end{corollary}

\section{Computing $R_K$ in special cases}

\begin{proposition}
Let $K$ be an imaginary quadratic field.  Then $R_K = \Z + 2 \oo_K$.
\end{proposition}

\begin{proof}
For $\e,\delta\in U_K$ let 
$$
S_{\e,\delta} := \{x \in \oo_K : \delta - 1 \equiv (\e-1)x \pmod{(\e-1)^2}\}.
$$
Then by definition, 
\begin{equation}
\label{capcup}
R_K := \bigcap_{\e\in U_K} \bigl(\bigcup_{\delta\in U_K} S_{\e,\delta}\bigr).
\end{equation}
One checks easily that 
$S_{1,1} = \oo_K$, $S_{-1,1} = 2\oo_K$, and $S_{-1,-1} = 1 + 2\oo_K$.
When $U_K = \{-1,1\}$, the proposition now follows from \eqref{capcup}.

Now suppose $K = \Q(i)$. In this case 
\begin{gather*}
S_{i,1} = S_{i,-1} = (1+i)\oo_K, \quad S_{i,i} = S_{i,-i} = 1 + (1+i)\oo_K, \\
S_{-i,1} = S_{-i,-1} = (1+i)\oo_K, \quad S_{-i,i} = S_{-i,-i} = 1 + (1+i)\oo_K, \\
S_{-1,i} = S_{-1,-i} = \emptyset.
\end{gather*}
Again it follows from \eqref{capcup} that $R_{\Q(i)} = \Z + 2\Z[i]$.
The proof when $K = \Q(\sqrt{-3})$ is similar.
\end{proof}

The following theorem will be used to study the case of unit groups of rank 1.
Its proof can be found in \cite{EG}.

\begin{theorem}[Corollary 6.1.2 of \cite{EG}]
\label{le:finmany}
Suppose $K$ is a number field.  Then the equation $x_1+x_2+x_3=1$ has only finitely 
many solutions with all $x_i \in U_K - \{1\}$.
\end{theorem}

\begin{lemma}
\label{enmo}
Suppose $K$ is a number field, $\e \in U_K$ is a unit of infinite order, 
and $a, b \in \Z$.  Let $g := \gcd(a,b)$.  
\begin{enumerate}
\item
If $b \mid a$ then $(\e^b-1) \mid (\e^a-1)$ and $(\e^a-1)/(\e^b-1) \equiv a/b \pmod{(\e^b-1)}$.
\item
The ideal of $\oo_K$ generated by $\e^a-1$ and $\e^b-1$ is $(\e^g-1)\oo_K$.
\item
If $(\e^b-1) \mid (\e^a-1)$ then $(\e^b-1)/(\e^g-1) \in U_K$.
\end{enumerate}
\end{lemma}

\begin{proof}
Assertion (1) follows from the polynomial identity $X^n-1 = (X-1)\sum_{i=0}^{n-1}X^i$ 
applied with $X := \e^b$, $n := a/b$.  

Let $\d$ be the ideal of $\oo_K$ generated by $\e^a-1$ and $\e^b-1$.  
By (1) we have that $\d \subset (\e^g-1)\oo_K$.  Fix $x, y \in \Z$ such that $ax - by = g$.
Using (1) again we see that $\d$ contains
$$
(\e^{ax}-1) - (\e^{by}-1) = \e^{ax} - \e^{by} = \e^{by}(\e^{ax-by}-1) = \e^{by}(\e^g-1).
$$ 
Since $\e$ is a unit, we have $\e^g-1 \in \d$, which completes the proof of (2).

By (1) and (2) we have that $(\e^b-1)/(\e^g-1)$ 
and $(\e^a-1)/(\e^g-1)$ are relatively prime integers of $K$.  
If in addition $(\e^b-1)/(\e^g-1)$ divides $(\e^a-1)/(\e^g-1)$, then 
$(\e^b-1)/(\e^g-1)$ must be a unit.  This proves (3).
\end{proof}

\begin{proposition}
\label{Nprop}
Suppose $K$ is a number field, and $u \in U_K$ is a unit of infinite order.
There is an $N \in \Z_{>0}$ such that for every $a, b \in \Z_{>0}$, and every $n > N$, 
the unit $\e := u^n$ satisfies
$$
(\e^b-1) \mid (\e^a-1) \iff b \mid a.
$$
\end{proposition}

\begin{proof}
The ``$\Longleftarrow$'' implication holds for every $n$ by Lemma \ref{enmo}(1).

Consider the set
$$
S := \{d \in \Z_{\neq 0} : \text{$\exists x_2, x_3 \in U_K-\{1\}$ such that $u^d + x_2 + x_3 = 1$}\}.
$$
By Theorem \ref{le:finmany}, $S$ is finite, and we set $N := \sup(S)$.  Fix $n > N$ and 
let $\e = u^n$.

Suppose $(\e^b-1) \mid (\e^a-1)$.  By Lemma \ref{enmo}(3), it follows 
that $\delta := (\e^b-1)/(\e^g-1) \in U_K$, where $g := \gcd(a,b)$.
Thus we have a solution of the unit equation
$$
\e^b - \delta\e^g + \delta = 1.
$$
We have $\e^b = u^{nb}$ with $nb \ge n > N$, so by definition of $S$ we must have
either $-\delta\e^g = 1$ or $\delta = 1$. Suppose $-\delta \e^g=1$.  
Since $u$ has infinite order, $\e^{b+g} \ne 1$. At the same time from the definition of $\delta$ we get that
$$
\e^g-1=-\delta\e^g(\e^g-1) = \e^g(1-\e^b) = \e^g-\e^{b+g}
$$
and $\e^{b+g}=1$.  This contradiction implies that $-\delta\e^g \ne 1$.  Therefore $\delta=1$, i.e., 
$\e^b-1 = \e^g-1$.  Again using that $u$ has infinite order we conclude that $g=b$, so $b \mid a$.
\end{proof}

\begin{proposition}
\label{le:rankone}
Let $K$ be a number field and let $W$ be a subgroup of $U_K$ such that $\rank W=1$.  
Then $R_{K,W}=\Z$.
\end{proposition}

\begin{proof}
Let $M := \{u^n : u \in W\}$, where $n$ is chosen large enough so that $M$ is 
torsion-free, and $n$ is larger than the $N$ of Proposition \ref{Nprop}.
Applying Lemma \ref{le:smaller} with this $M$ (and $\a = \oo_K$) shows 
that $nR_{K,W} \subset R_{K,M}$, so to 
prove the proposition it is enough to show that $R_{K,M} = \Z$.

Fix a generator $\e$ of $M$, and suppose $x \in R_{K,M}$.  Choose $m \in \Z_{>0}$, 
and choose $b \in \Z_{>0}$ such that $\e^b \equiv 1 \pmod{m}$.
By definition of $R_{K,M}$, for some $a \in \Z$ we have
$$
\e^a-1 \equiv (\e^b-1)x \pmod{(\e^b-1)^2}.
$$
In particular $(\e^b-1) \mid (\e^a-1)$, so by Proposition \ref{Nprop}
we conclude that $b \mid a$.  Then 
$$
x \equiv (\e^a-1)/(\e^b-1) \equiv a/b \pmod{(\e^b-1)}
$$
by Lemma \ref{enmo}(1), 
so $x-a/b \in m\oo_K$.  In particular, $x$ is divisible by $m$ in the free 
$\Z$-module $\oo_K/\Z$.  Since $m$ is arbitrary, we conclude that 
$x = 0$ in $\oo_K/\Z$, i.e., $x \in \Z$.
\end{proof}

\begin{lemma}
\label{le:notnew}
Let $p$ be a rational prime, and $K$ a number field. 
Suppose $j \in \Z_{> 0}$, and $\e\in U_K - \{1\}$ satisfies $\e\equiv 1 \pmod{p^{j+1}}$. 
Then there exists $x \in \oo_K$ such that for every integer $n$ prime to $p$, the congruence 
\begin{equation*}
\delta-1 \equiv (\e-1)p^j nx \pmod{p^{j+1}(\e-1)\oo_K}
\end{equation*}
has {\em no} solutions $\delta\in U_K$. 
\end{lemma}

\begin{proof}
Let $\a=\prod_{\pp|p}\pp^{\ord_{\pp}(\e-1)}$, where $\pp$ runs through the 
primes of $K$ above $p$.  Note that $p^{j+1}$ divides $\a$ by our choice of $\e$.  
Define groups 
$$
W_{\a} := \{u \in U_K : u \equiv 1 \pmod{p^j\a}\}, 
   \quad  V_{\a} := p^j\a/p^{j+1}\a.
$$
Then $W_{\a}  \subset U_K$ is a free abelian group of rank strictly 
less than $[K:\Q]$ (there are no nontrivial roots of unity congruent to $1$ modulo $p\a$), and 
$V_{\a}$ is an $\F_p$-vector space of dimension $[K:\Q]$.
There is a natural homomorphism 
\[
\lambda_{\a} : W_{\a} \longrightarrow V_{\a}
\]
 defined by 
$$
\lambda_{\a}(u) = u-1 \pmod {p^{j+1}\a},
$$
and it follows that the cokernel of $\lambda_{\a}$ has dimension at least one.

The map  
$$
\text{$\oo_K/p\oo_K \longrightarrow p^j\a/p^{j+1}\a = V_\a$ \quad defined by \quad $x \mapsto (\e-1)p^j x$}
$$ 
is an isomorphism, so we can choose $x \in \oo_K$ such that for every integer $n$ 
prime to $p$ the product $(\e-1)p^j nx$ 
is not in the image of $\lambda_\a$.  But not being in the image of $\lambda_\a$ 
means precisely that there is no $\delta \in U_K$ such that 
\begin{equation*}
\delta-1 \equiv (\e-1)p^j nx  \pmod{p^{j+1}(\e-1)\oo_K}.
\end{equation*}
\end{proof}

\begin{definition}
Let $p$ be a prime. Let $K/\Q$ be a possibly infinite degree algebraic extension.  
Then we will say that $K$ is  {\em$p$-finite} if $\ord_p[K:\Q] < \infty$,
or equivalently if for some 
number field $L \subset K$ no finite extension of $L$ in $K$ 
has degree divisible by $p$.
\end{definition}

\begin{remark} 
An  algebraic extension of  $\Q$ is not big  
(cf. Definition \ref{big}) if and only if there exists a prime $p$ for which it is 
$p$-finite.
\end{remark}

\begin{lemma}
\label{cor:notthere}
Suppose $K$ is a number field such that $U_K$ is infinite. 
Let $L$ be an algebraic extension of $K$ that is not big.
Then for every $m \in \Z_{>0}$ we have that $m\oo_K \not \subset R_{L}$.  
\end{lemma}

\begin{proof}
Since $L$ is not big, we can choose a rational prime $p$ such that $L$ is $p$-finite.
Fix $m \ge 1$, and define 
$$
j := 1 + \ord_p(m) + \sup\{\ord_p[K':K] : \text{$K'$ is a finite extension of $K$ in $L$}\}
$$
which is finite by assumption.  
Since $U_K$ is infinite, we can fix $\e \in U_K - \{1\}$ such that $\e \equiv 1 \pmod{p^{j+1}}$.
Let $x \in \oo_K$ be as in Lemma \ref{le:notnew} with this choice of $p$, $K$, $\e$, and this $j$.

Suppose $mx \in R_L$.  Then there exists $\delta \in U_L$ such that 
$$
\delta-1 \equiv (\e-1)mx \pmod {(\e-1)^2 \oo_{L}}.
$$
By Lemma \ref{norm}, 
$$
N_{K(\delta)/K}\delta-1 \equiv (\e-1)[K(\delta):K]mx \pmod {(\e-1)^2 \oo_{K}}.
$$
Write $[K(\delta):K]m = np^s$, with $n$ prime to $p$.  Then $s < j$, and 
by Lemma \ref{sumform}
$$
(N_{K(\delta)/K}\delta)^{p^{j-s}}-1 \equiv (\e-1)[K(\delta):K]p^{j-s}mx \equiv (\e-1)p^j nx \pmod {(\e-1)^2 \oo_{K}}.
$$ 
This contradicts our choice of $x$ (Lemma \ref{le:notnew}), so we must have $mx \notin R_L$.
\end{proof}

\begin{theorem}
\label{cor:oddp}
Let $K$ be an algebraic extension of $\Q$  that is not big.   
Then either $R_{K}=\Z$ or $R_K$ is an order in some imaginary quadratic field $F$ contained in 
$K$.
\end{theorem}

\begin{proof}
Let $x \in R_{K}$, and assume that 
that $\Q(x)$ is neither $\Q$ nor an imaginary quadratic field.
In particular $U_{\Q(x)}$ is infinite.  
Since $R_{K}$ is a ring we have $\Z[x] \subset R_{K}$, so there 
exists $m \in \Z_{>0}$ such that $m\oo_{\Q(x)}\subset R_{K}$. 
Thus contradicts Lemma \ref{cor:notthere}.    

Hence for any $x \in R_{K}$ we have that $\Q(x)=\Q$ or $\Q(x)$ is an imaginary quadratic 
field.  If $x_1, x_2 \in R_{K}$ generate two different imaginary quadratic fields 
$\Q(x_1), \Q(x_2)$, then $R_{K}$ contains an element 
$x$ such that $\Q(x)=\Q(x_1,x_2)$ leading to a contradiction again.  
Thus either $R_{K}=\Z$ or $K$ is an imaginary quadratic field.
\end{proof}

\begin{corollary}
\label{cor:notbig}
Suppose $K$ is not big.  If $K$ does not contain an 
imaginary quadratic field, or $K$ is a CM-field, then $R_{K}=\Z$.
\end{corollary}

\begin{proof}
The case of a field $K$ without an imaginary quadratic subfield 
follows directly from Theorem \ref{cor:oddp}.  
The CM-field case follows from Theorem \ref{cor:oddp} and Proposition \ref{cor:real1}.
\end{proof}

\begin{corollary}
\label{cor:deg4}
Let $K$ be a degree 4 extension of $\Q$.  Then $R_K=\Z$.
\end{corollary}

\begin{proof}
If $K$ does not contain an imaginary quadratic field, then the corollary follows from 
Corollary \ref{cor:notbig}.  If $K$ does contain an imaginary quadratic field, then 
$K$ has no real embeddings, so $\rank U_K = 1$ and the corollary follows from 
Proposition \ref{le:rankone}.
\end{proof}

Given a non-big field $K$ that contains more than one imaginary quadratic field, 
at most one of these fields can be the fraction field of $R_K$.  
This lack of symmetry, together with Corollaries \ref{cor:notbig} and \ref{cor:deg4}, 
leads us to the following conjecture.

\begin{conjecture}
Let $K$ be a non-big field that is not imaginary quadratic.  Then $R_K=\Z$.
\end{conjecture}

We conclude this section by describing a class of (necessarily big) fields $K$ 
such that $R_{K}=\oo_{K}$.  

\begin{lemma}
\label{le:exist}
Suppose $K$ is a number field, $I \subset \oo_K$ is a nonzero ideal, and $\beta \in \oo_K$ is 
relatively prime to $I$.  Then there is a finite extension $L/K$ and a unit 
$\delta \in U_L$ such that $\delta \equiv \beta \pmod{I\oo_L}$.  
Further, the field $L$ can be taken to be generated by a root of a polynomial 
of the form $X^d + a X^{d-1} + b$, where $d \in \Z_{>0}$ and $a, b \in \oo_K$.
\end{lemma}

\begin{proof}
Without loss of generality (replacing $I$ by a multiple if necessary) we can assume that 
$I$ is principal.  Fix a generator $\mu$ of $I$. 
 
Let $d$ be the order of $-\beta$ in $(\oo_K/I)^\times/\{\text{$\image$ of $U_K\}$}$.  
Then we can fix a unit $u \in U_K$ such that $(-\beta)^d \equiv u \pmod{I}$.

Consider the polynomial 
$$
f(X) := X^d + a X^{d-1} + b \in \oo_K[X]
$$
where $a, b \in \oo_K$
will be chosen later.  Let $\rho$ be a root of $f$, let $L = K(\rho)$, 
and define another monic polynomial
$$
g(X) := \mu^d f((X-\beta)/\mu) \in \oo_K[X].
$$
Then $\delta := \mu\rho+\beta \in \oo_L$ is a root of $g$, and $\delta \equiv \beta \pmod{I\oo_L}$.

We claim that there are elements  $a, b \in \oo_K$ such that $g(0) =u$ is a unit.  In that case 
we will have that $\delta$ is a unit, so $\delta$ will have all the desired properties, 
concluding the proof of Lemma \ref{le:exist}.   

To prove the claim we have 
\begin{multline}
\label{1}
g(0) = \mu^d f(-\beta/\mu) = (-\beta)^d + a \mu(-\beta)^{d-1} + \mu^d b	\\
   = u + ((-\beta)^d -u) + a \mu(-\beta)^{d-1} + b\mu^d. 
\end{multline}
Since $\beta$ is relatively prime to $\mu$, we have
$$
\{a \mu(-\beta)^{d-1} + b\mu^d : a, b \in \oo_K\} = \mu\oo_K.
$$
In particular since $(-\beta)^d \equiv u \pmod{\mu}$, we can find $a, b \in \oo_K$ 
such that 
$$
a \mu(-\beta)^{d-1} + b\mu^d = u - (-\beta)^d.
$$
By \eqref{1} we have $g(0) = u$.  This completes the proof of the claim, and the lemma.
\end{proof}

\begin{proposition}
\label{onlyQbar}
Let ${F}$ be a field of algebraic numbers such that for every $d \in \Z_{>0}$ and every 
$a, b \in \oo_{{F}}$, the polynomial 
$X^d + a X^{d-1} + b$ has a root in ${F}$.  Then $R_{F}=\oo_{{F}}$.
\end{proposition}

\begin{proof}
Suppose $x \in \oo_{F}$ and $\e \in U_{F}, \e \ne 1$.  
Applying Lemma \ref{le:exist} with 
$$
K := \Q(x,\e), \quad \beta := 1 + x(\e-1), \quad I := (\e-1)^2
$$
we get a unit $\delta \in U_{F}$ satisfying
$$
\delta-1 \equiv x(\e-1) \pmod{(\e-1)^2}
$$
(note that $\delta \in U_{F}$ thanks to our assumption on ${F}$).
It follows that $x \in R_{F}$.
\end{proof}

\begin{remark}
The field $\Qbar$ satisfies the hypothesis of Proposition \ref{onlyQbar}; 
we don't know of any other field that does. 

Among big fields  of algebraic numbers $K$, which of them have the property 
that the field of fractions of $R_K$ is $K$? 

For a big field $K$ 
denote by $K_{\{1\}}$ the field of fractions of $R_K$. We have $K_{\{1\}}\subset K$ 
by Proposition \ref{prop:isaring}(2), and 
we can iterate this operation, denoting by $K_{\{i+1\}}$ the field of 
fractions of the ring $R_{K_{\{i\}}}$.  This tower 
$$
{\Q}\subset\cdots K_{\{i\}}\subset  K_{\{i-1\}}\dots \subset  K_{\{1\}} \subset  K
$$
gives us a first order definition (involving $i+1$ universal quantifiers) 
of the ring of integers of $ K_{\{i\}}$ in $\oo_K$. 

Are there examples of 
big fields $K$ where this descending sequence of fields doesn't stabilize? 
That is, such that $ K_{\{i+1\}}$ is a proper subfield of $ K_{\{i\}}$ for all $i$?
\end{remark}

\section{Proof of the main theorem}
\label{mainpf}

%In this section we prove our main result (Theorem \ref{thm:main}).
Before proving  our main result, Theorem \ref{thm:main}, we need a lemma to
handle the case where our formula possibly defines an order 
in an imaginary quadratic field.  In what follows we suppress---i.e., we  understand implicitly---equations and/or variables 
needed to say that an element of a ring is a unit or that a congruence holds over a ring.

Let 
$$
D(w) := 3^{12} w^4 (w-1)^4(w-2)^4 \in \Z[w]
$$ 
be the quantity denoted $D(2,w)$ in \cite[\S11]{MRSh22}.

\begin{lemma}[Cf. \S 3 of \cite{Den1}]
\label{prel}
Fix  a non-square positive integer $d \equiv 3 \bmod 4$.
For $F$ an imaginary quadratic field let 
$\calS_F = \calS_{F,d} (w, s_1, s_2, t_1, t_2, u_1, u_2, v_1, v_2)$ denote the system of equations 
and congruences \eqref{eq:delta1} through \eqref{eq:w}:
  
\begin{gather}
\label{eq:delta1}
\delta_i := s_i+t_i\sqrt{d} \in U_{F(\sqrt{d})}, \quad s_i, t_i \in \oo_F, \quad i=1,2,\\ 
\label{eq:square} 
\varepsilon_i:=\delta_i^2=u_i+v_i\sqrt{d}, \quad u_i, v_i \in \oo_F, \quad i=1,2,\\
\label{eq:mod32}
\varepsilon_1 \equiv 1 \pmod{d D(w)\oo_F}, \\
\label{eq:w}
\varepsilon_2-1 \equiv w(\varepsilon_1-1) \pmod{(\varepsilon_1-1)^2}.
\end{gather}

Then  ${\calS}_{F,d}$ constitutes a uniform existential definition  of $\Z$ over the ring of integers $ \oo_F$ of any 
imaginary quadratic field $F$.
\end{lemma}

\begin{proof}
Suppose $(w, s_1, s_2, t_1, t_2, u_1, u_2, v_1, v_2) \in \oo_F^9$ is a zero of $\calS$.
Then $\e_1$ and $\e_2$ are squares of units by \eqref{eq:delta1}, \eqref{eq:square}, 
and congruent to $1 \pmod{d}$ by \eqref{eq:mod32}, \eqref{eq:w}, so by the same argument as in the proof of
Proposition \ref{cor:real1} we conclude that $\e_1, \e_2 \in \Z[\sqrt{d}]=\oo_{\Q(\sqrt{d})}$. 
Now by Corollary 11.12 of \cite{MRSh22} applied to the extension $F(\sqrt{d})/\Q(\sqrt{d})$ and the ideal $I=(\e-1)$ we have that $w \in \oo_F \cap \Z[\sqrt{d}]=\Z$.

Conversely, suppose $w \in \Z$.  Let $s_1,t_1 \in \Z$ be such that 
$\delta_1 \in U_{\Q(\sqrt{d})} \subset U_{F(\sqrt{d})}$ and 
\[
s_1\equiv 1, \quad t_1 \equiv 0 \pmod{dD(w)}.
\]
Let $u_1, v_1 \in \Z$ be such that $\varepsilon_1=\delta_1^2$.  
Then \eqref{eq:delta1}- \eqref{eq:mod32} hold for  $i=1$.  
 
 Let $s_2, t_2\in \Z$ be such that $\delta_2=\delta_1^w$.  Let $u_2, v_2 \in \Z$ be such that $\varepsilon_2=\delta_2^2$.  Then \eqref{eq:delta1}- \eqref{eq:mod32} hold for  $i=2$.
\end{proof}
\begin{remark}
If we replace  $\oo_F$ by $\Z$  then for any integer $w$ we can still satisfy all equations $\calS_F$ in the proof 
of Lemma \ref{prel}.
\end{remark}
\begin{remark}
Denef's proof in \cite{Den1}, while not formally asserting {\it uniformity} across imaginary quadratic fields, can be easily converted into a uniform proof.  

\end{remark}
\begin{theorem}
\label{thm:main}
Let $\calA$ be the collection of all non-big fields of algebraic numbers.
There exists a first-order formula of the form 
``$\forall \forall\exists \ldots \exists$'' uniformly defining $\Z$ over the 
rings of integers of all fields in $\calA$.  
Thus, the first-order theory of these rings is undecidable.
\end{theorem}

Before beginning the proof, let  $K$ be any field of algebraic numbers 
and consider the following  first-order formula of the form 
``$\forall \forall\exists \ldots \exists$''  and the set $Z_K \subset \oo_K$ it defines: 

\begin{definition}
\label{formula} 
Define $Z_K \subset \oo_{K}$ to be the 
set of all $w \in \oo_K$ satisfying  
\[
 \forall \varepsilon \in U_{K}\setminus \{1\},
 \exists \mu, \mu_1, \mu_2, \nu_1,\nu_2, \gamma_1, \gamma_2,\tau_1, \tau_2 \in U_{K}, s_1, s_2, t_1, t_2, u_1, u_2, v_1, v_2 \in \oo_{K} 
 \]
such that
\begin{gather}
\label{eq:z1} w(\varepsilon-1) \equiv \mu-1 \pmod{(\varepsilon-1)^2},\\
u_i(\varepsilon-1) \equiv \mu_i-1  \pmod{(\varepsilon-1)^2}, \quad i=1,2, \\
v_i(\varepsilon-1) \equiv \nu_i-1 \pmod{(\varepsilon-1)^2}, \quad i=1,2, \\
s_i(\varepsilon-1) \equiv  \sigma_i-1 \pmod{(\varepsilon-1)^2}, \quad i=1,2, \\
\label{eq:Z} t_i(\varepsilon-1) \equiv \tau_i-1 \pmod{(\varepsilon-1)^2}, \quad i=1,2, \\
\label{eq:S} \calS_F(w, s_1, s_2, t_1, t_2, u_1, u_2, v_1, v_2),
\end{gather}
with $\calS$ as in the proof of Lemma \ref{prel}.
\end{definition}

\begin{remark}
Unfortunately, we cannot quantify over the group of units.  Therefore, instead of using ``$\forall \varepsilon \in U_{K}\setminus \{1\}$'' in Definition \ref{formula}  we will quantify over all $\varepsilon, \beta \in \oo_K$ by requiring nothing if  $\varepsilon \beta$ is not equal to 1 or if $\varepsilon =1$ and requiring \eqref{eq:z1}--\eqref{eq:S} if  $\varepsilon \beta=1$.  More formally, we have the following definition of $\Z$:
\[
\forall \varepsilon \forall \beta \in \oo_K: 
(\varepsilon\beta=1 \land [(\e = 1) \lor ((\e \ne 1) \land \ldots))]) \lor (\varepsilon \beta \ne 1),
\]
where we insert equations \eqref{eq:z1}--\eqref{eq:S} in place of ``$\ldots$''.  
Essentially we need a second universal quantifier if $\varepsilon$ 
is not a unit.  Having a uniform existential definition of non-units would 
eliminate the need  for the second universal quantifier, but we do not know 
how to construct such a definition or even if it exists.
 \end{remark}
 
\begin{proof}[Proof of Theorem \ref{thm:main}]
Let $K$ be non-big, and $Z_K$ as in Definition \ref{formula}.  We will show that $Z_K = \Z$.

Suppose $w \in Z_K$.  Then by Theorem \ref{cor:oddp} we have that either 
$$
\text{$w, u_1, u_2, v_1, v_2,  s_1, s_2, t_1, t_2 \in \Z$ 
\quad or \quad
$w, u_1, u_2, v_1, v_2,  s_1, s_2, t_1, t_2 \in \oo_F$}
$$
for some imaginary quadratic field $F$.  In either case equations \eqref{eq:S} imply $w \in \Z$.

Conversely, if $w \in \Z$, then the proof of Lemma \ref{prel} shows that 
the equations of \eqref{eq:S} can be satisfied with all variables ranging over $\Z$.  
Therefore equations \eqref{eq:z1}--\eqref{eq:S} can be satisfied over $\oo_{K}$, so $w \in Z_K$.
\end{proof}

\begin{remark}
Let $\calA$ denote the set of all algebraic extensions $K$ of $\Q$ such that the 
set $Z_K$ of Definition \ref{formula} is equal to $\Z$.  Theorem \ref{thm:main} 
implies that $\calA$ contains all non-big fields, including in particular all number fields.  
This leads us to the following questions.

\begin{question}
Does $\calA$ contain any big fields?
\end{question} 

\begin{question}
Does $\calA$ contain all abelian fields?
\end{question} 
\end{remark}
\begin{remark}
The construction we used above can be adapted in the case of number fields to produce a definition of $\Z$ of the form $\forall \exists \ldots \exists$.  Unfortunately this definition, as in the case of the definition in this form constructed by J. Robinson,  is uniform only across classes of number fields of bounded degree.

\end{remark}

\section{Non-existence of uniform existential or purely universal definitions of $\Z$}

Recall the definition of ``uniform'' (Definition \ref{def:uni}) and consider the following definition.
\begin{definition}
Let $R$ be a ring.  Let $A \subset R$ and assume there exists a polynomial $p(t,x_1,\ldots, x_k) \in R[t,x_1, \ldots, x_k]$ such that 
\[
A=\{t \in R|\exists  x_1,\ldots,x_k \in R^k: p(t,x_1,\ldots,x_k)=0\}.
\] 
Then $p(t,\bar x)$ is called a diophantine definition of $A$ over $R$.
\end{definition}

\begin{proposition}
\label{thm:uni} Let $K$ be a number field, $L$ an algebraic extension of $K$, 
and $A$ a subset of $\oo_L$.  
If there is a diophantine definition of $A\cap\oo_L'$ over $\oo_{L'}$ uniform across the 
rings of algebraic integers $\oo_{L'}$ for all finite extensions $L'$ of $K$ in $L$, 
then there exists a diophantine definition of $A$ over $\oo_L$.
\end{proposition}

\begin{proof}
Let $p(t,\bar x) \in \oo_{K}[t,\bar x]$ be a diophantine definition of $\oo_{L'} \cap A$ 
for all finite extensions $L'$ of $K$ in $L$.  
Let $t \in A \subset \oo_L$.  Let $L'=K(t)$.  Then there exists an $m$-tuple 
$\bar x \in \oo_{L'}^m\subset \oo_L^m$ such that $p(t,\bar x)=0$. 

Now suppose there exists $t \in \oo_L \setminus A$ such that for some 
${\bar y} \in \oo_L$ we have that $p(t,{\bar y})=0$.  Let $L'=K(t,\bar y)$.  
Then  $p(t, \bar x)$ is not a diophantine definition of $A \cap \oo_{L'}$ 
over $\oo_{L'}$, contradicting our assumption on $p(t,\bar x)$.  
Thus, for any $t \in \oo_L$ we have that there exists ${\bar x} \in \oo_L$ 
such that $p(t,\bar x)=0$ if and only if $t \in A$.  
\end{proof}

\begin{corollary}
Let $K$ be a number field and $\calL$ the collection of all finite 
extensions of $K$.  Then there is no uniform purely existential 
or purely universal definition of $\oo_K$ across rings of integers of elements of $\calL$.
\end{corollary}

\begin{proof}
We apply Proposition \ref{thm:uni} to $A:=A_1 =\oo_K$ and 
$A:=A_2=\oo_{\bar K}\setminus \oo_K$, with $F = \bar{K}$, the algebraic closure of $K$.  
If $A_1$ has a uniform diophantine definition across fields in $\calL$, then $\oo_K$ has a diophantine definition over $\oo_{\bar K}$.  If $A_2$ has a uniform diophantine definition across rings of integers of fields in $\calL$, then $\oo_{\bar K} \setminus \oo_K$ has a diophantine definition over $\oo_{\bar K}$ or, equivalently, $\oo_K$ has a purely universal definition over $\oo_{\bar K}$.  In either case $\oo_K$ has a first-order definition over $\oo_{\bar K}$.  But by the result of J. Robinson (\cite{Rob1}) we know that the first-order theory of any number field is undecidable and by a result of van den Dries (\cite{Dries4}) we know that the first-order theory of $\oo_{\bar \Q}$ is decidable.  Hence neither $A_1$, nor $A_2$ can have a uniform existential definition across elements of $\calL$.  Therefore, there is no uniform purely universal definition of $\oo_K$ across rings of integers of  elements of $\calL$.
\end{proof}

Applying the discussion above to $K=\Q$. we get:
\begin{corollary}
\label{thm:nonu}
There is no existential or purely universal definition of $\Z$ uniform across all rings of integers of number fields.
\end{corollary}
Finally we note that in view of the results above, the uniform definition of $\Z$ across all number fields we give in Theorem \ref{thm:main} is the best possible as far as the arithmetic hierarchy is concerned: the statement of that theorem would be false if we replaced  ``$\forall \forall\exists \ldots \exists$''   by ``$\exists \ldots \exists$".

\end{document}